\documentclass[12pt]{article}
\newlength\titlebox 
\usepackage{xurl}
\usepackage[colorlinks,citecolor=blue,urlcolor=blue, linkcolor=blue, bookmarks=false,hypertexnames=true]{hyperref}
\usepackage{url}

\newcommand\outauthor{
	\begin{tabular}[t]{c}
		\bf\@author
	\end{tabular}}
	
\providecommand{\keywords}[1]
{\small\textbf{Keywords:} #1
}

\usepackage[utf8]{inputenc}
\usepackage[english]{babel}
\usepackage{amsthm, mathtools, amsmath, enumerate}
\usepackage{graphicx}
\usepackage{float}
\usepackage{subcaption, caption}
\usepackage{xcolor}
\usepackage{cleveref}

\usepackage{thmtools} 
\usepackage{thm-restate}

\newtheorem{Theorem}{Theorem}
\newtheorem{Def}[Theorem]{Definition}

\newtheorem{Lemma}[Theorem]{Lemma}
\newtheorem{Cor}[Theorem]{Corollary}
\newtheorem{Obs}[Theorem]{Observation}
\newtheorem{Con}[Theorem]{Conjecture}
\newtheorem{Q}[Theorem]{Question}

\theoremstyle{definition}
\newtheorem{Remark}[Theorem]{Remark}
\newtheorem*{acknowledgement}{Acknowledgements}
\newtheorem*{open}{Open access statement}

\DeclareMathOperator{\Lf}{Leaf}

\begin{document}
\title{Recognizing trees from incomplete decks}

\author{
Gabriëlle Zwaneveld\\
University of Amsterdam, The Netherlands\\
\tt{gabriellezwaneveld@gmail.com}\\
}

\maketitle

\small

\begin{abstract}
For a given graph, the unlabeled subgraphs $G-v$ are called
the cards of $G$ and the deck of $G$ is the multiset $\{G-v: v \in V(G)\}$.
Wendy Myrvold [Ars Combinatoria, 1989] showed that a non-connected graph and a connected graph both on $n$ vertices have at most $\lfloor \frac{n}{2} \rfloor +1$ cards in common and she found (infinite) families of trees and non-connected forests for which this upper bound is tight. Bowler, Brown, and Fenner [Journal of Graph Theory, 2010] conjectured that this bound is tight for $n \geq 44$. In this article, we prove this conjecture for sufficiently large $n$. The main result is that a tree $T$ and a unicyclic graph $G$ on $n$ vertices have at most $\lfloor \frac{n}{2} \rfloor+1$ common cards. Combined with Myrvold's work this shows that it can be determined whether a graph on $n$ vertices is a tree from any $\lfloor \frac{n}{2}\rfloor+2$ of its cards.

Based on this theorem, it follows that any forest and non-forest also have at most $\lfloor \frac{n}{2} \rfloor +1$ common cards. Moreover, we have classified all except finitely many pairs for which this bound is strict. Furthermore, the main ideas of the proof for trees are used to show that the girth of a graph on $n$ vertices can be determined based on any $\frac{2n}{3} +1$ of its cards. Lastly, we show that any $\frac{5n}{6} +2$ cards determine whether a graph is bipartite.
\end{abstract}

\section{Introduction}
Let $G=(V, E)$ be a (simple) graph on $n$ vertices. The \textit{card} of $G$ \textit{corresponding to} a vertex $v \in V$ is the graph $G-v$. This is the subgraph of $G$ where the vertex $v$ and all edges incident to $v$ are removed. The \textit{deck} of $G$ is the multiset of the $n$ unlabeled cards $G-v$. Different cards of $G$ can be isomorphic, and  the number of cards in the deck of $G$ isomorphic to a graph $C$ is, by definition, equal to the number of vertices $v \in V(G)$ such that $G-\{v\} \cong C$. 

A natural question is whether the deck of cards determines the graph $G$. A graph $G$ is called \textit{reconstructible} if it is uniquely (up to isomorphism) determined by its deck.
 In 1942 Paul Kelly posed the following famous conjecture:

\begin{Con}[The Reconstruction Conjecture (Kelly \cite{kelly1942isometric,kelly1957congruence}, Ulam \cite{ulam1960collection})]
Any graph having at least 3 vertices is
reconstructible.
\end{Con}

The condition that the graph has at least $3$ vertices is necessary, as both graphs on two vertices ($K_2$ and $2 K_1$) have a deck of two cards $K_1$. Hence, in this article, we will only consider graphs that have at least 3 vertices, even when it is not explicitly stated.

Although this conjecture is still open, we know that many graph properties can be deduced from the deck. For example, Kelly \cite{kelly1957congruence} used a counting argument to show that it is possible to determine the number of subgraphs of $G$ isomorphic to $F$, for all graphs $F$ with strictly fewer vertices than $G$. Also, the characteristic polynomial, the number of Hamiltonian cycles, and the number of spanning trees \cite{tutte1979all} can be deduced from the deck. Therefore, we use the following definition: A graph property is called \textit{recognizable} if it can be deduced from the deck of a graph.

Furthermore, mathematicians have shown that all members of certain families of graphs are reconstructible. One of the first classes for which this was shown was trees in 1957 by Paul Kelly \cite{kelly1957congruence}. Since then, much research has been conducted to determine the number of cards one needs to deduce whether a graph is a tree. For instance, Wendy Myrvold \cite{myrvold1990allytree} showed that every tree on at least 5 vertices has 3 special cards from which we can reconstruct the original tree. She \cite{myrvold1989ally} also found several families of examples of trees $G$ and disconnected graphs $H$ such that the decks of these two graphs contain the same subset $\lfloor\frac{n}{2}\rfloor+1$ cards and showed that any $\lfloor\frac{n}{2}\rfloor+2$ cards in the deck of $G$ determine whether $G$ is connected.

At the same time, it is difficult to determine the precise number of edges of $G$ based on a subdeck of $G$. The current best-known result is that for $n$ sufficiently large and $k \leq \frac{1}{20}\sqrt{n}$, the number of edges of a graph is recognizable from any $n - k$ cards. This result was found in 2021 by Groenland, Guggiari, and Scott \cite{groenland2021size}. Moreover, Bowler, Brown, and Fenner \cite{bowler2010families} found examples of graphs with a different number of edges that have $2\lfloor \frac{1}{3}(n-1)\rfloor$ cards in common. These results both indicate that in order to find the minimum number of cards needed to determine whether a graph is a tree, one can not simply try to reconstruct $e(G)$. Therefore, it is necessary to consider the number of overlapping cards between a tree $T$ and a connected non-tree $G$.

Bowler, Brown and Fenner \cite{bowler2010families} gave an example of
infinitely many pairs of $n$-vertex trees and $n$-vertex connected non-trees that have $\lfloor \frac{2}{5}(n+1)\rfloor$ cards in common, and based on their example they stated the following conjectures:

\begin{Con}[Bowler, Brown and Fenner \cite{bowler2010families}]
    For $n\geq 44$, the only  pair of graphs on $n$ vertices that have at least $\lfloor \frac{2}{5}(n+1)\rfloor$ common cards, where one is a tree and the other a connected non-tree, is the pair found in Theorem 3.6 of \cite{bowler2010families}.
\end{Con}

\begin{Con}[Bowler, Brown and Fenner \cite{bowler2010families}]
    For $n\geq 44$, it can be determined whether a graph is a tree from any $\lfloor\frac{n}{2}\rfloor+2$ of its cards.
\end{Con}

In the example from Theorem 3.6 of \cite{bowler2010families}, the tree is a caterpillar graph and the connected non-tree is a sunshine graph (a cycle with some leaves). This led to further investigation into the overlapping cards between these special subclasses of trees and connected non-trees. In his PhD thesis \cite{brown2008maximum}, Paul Brown showed that the maximum number of overlapping cards between a sunshine and a caterpillar graph is at most $\lfloor \frac{2}{5}(n+1)\rfloor$ and that bound is only attained by the family found by Bowler et al. 

The main result of this article is that for $n$ large enough the maximum overlapping cards between a tree and a connected non-tree is at most $\lfloor \frac{n}{2}\rfloor+1$. Moreover, it contains a remark that indicates how this proof can be adapted to prove a slightly better bound of $\frac{n}{2}-\frac{\sqrt{n}}{80}$.
Even though these are not the conjectured bound of $\lfloor \frac{2}{5}(n+1)\rfloor$, it is still sufficient to prove the following theorem:

\begin{restatable}{Theorem}{maintrees}\label{main thm trees}
For all $n \geq 5000$, it can be determined whether a graph with $n$ vertices is a tree based on any $\lfloor\frac{n}{2}\rfloor+2$ of its cards.
\end{restatable}

Section 2 contains a sketch of the main ideas in the proof overview. The details of the proof can be found in section \ref{proof Thm}. Thereafter, we apply the main theorem to prove a generalization about the number of common cards between forests and non-forests. This proof uses several cases, where in one case we need the main theorem. As the main theorem only holds for $n\geq 5000$, the following theorem about forests also only holds for $n\geq 5000$.

\begin{restatable}{Theorem}{mainforest}\label{main forest}
For $n \geq 5000$, it can be determined whether a graph with $n$ vertices is a forest based on any $\lfloor\frac{n}{2}\rfloor+2$ of its cards.
\end{restatable}

The last two sections contain applications of coloring techniques similar to the coloring used in the proof of Theorem \ref{main thm trees} to prove other results that are determined by the (non-)existence of certain cycles. By using these colorings, it can be shown that both the girth, the length of the shortest cycle, and bipartiteness, the existence of cycles of odd length, can be deduced even if there is a linear amount of missing cards.

\begin{restatable}{Theorem}{thmgirth}\label{thm girth}
   For a graph on $n$ vertices, any $\frac{2n}{3}+1$ cards are sufficient to determine its girth.
\end{restatable}

\begin{restatable}{Theorem}{thmbipartite}\label{thm bipartite}
For any graph on $n$ vertices, any $\frac{5n}{6}+2$ of its cards are sufficient to determine if it is bipartite.
\end{restatable}

The proofs of these last two results indicate that the method of coloring used here is not only applicable to the proof of the upper bounds for the number of cards one needs to deduce whether a graph is a tree, but that it also can be adapted to proving upper bounds for other properties that involve cycle recognition.

\section{Proof overview}
We will consider a (fixed) tree $T$ and a unicyclic graph $G$ both on $n \geq 3$ vertices that have $\lfloor \frac{n}{2}\rfloor +2$ common cards. Since every subgraph of $T$ is a forest,  every common card of $G$ and $T$ is also a forest. If $G$ has at least $2$ cycles, then we can easily see that $T$ and $G$ have at most 2 common cards. Hence, $G$ has exactly one cycle, so $G$ is \textit{unicyclic}. As this cycle can not be visible on any card, every common card corresponds to a vertex in the cycle of $G$. Hence, the length $L$ of the unique cycle of $G$ satisfies $L \geq \lfloor \frac{n}{2}\rfloor+2$. Moreover, as $e(G)=n=e(T)+1$ a common card $T-v=G-w$ satisfies $\deg(w)=\deg(v)+1$. As trees have many vertices of degree at most 2, we can use a counting argument to find a lower bound on the number of leaves of $G$. Together with the lower bound on the length of the cycle, this gives an upper bound for the number of vertices of degree $\geq 2$ in the branches of $G$.

Next, we will use the coloring in Figure \ref{Kleuring 1} to make a distinction between vertices on the cycle that lie close to large branches (the red vertices) and the vertices that lie far enough from large branches (the white vertices). Then, we use a counting argument to show that at least one common card corresponds to a white vertex $w$ in $G$. As the large branches lie `far enough' away from $w$, the longest path on this card will only be a little longer than the length of the cycle itself. As $G-w=T-v$ consists of one large component and some isolated vertices, this also gives an upper bound for the longest path in $T$.

Thereafter, we will use another coloring, see Figure \ref{Kleuring2}, to find a vertex $w \in V(G)$ such that $G-w = T-v$ consists of one large component and some isolated vertices and such that none of the vertices $v'$ in the `middle' of the longest path of $G-w = T-v$ can correspond to a common card of $T$ and $G$. We prove this by showing that every path in $T-v'$ that uses a branch connected to the `middle' of the longest path is somewhat shorter than the longest path itself, which implies that the longest path in $T-v-v'$ will be also somewhat shorter than $L-1$. As adding the vertex $v$ back increases the longest path by at most 2 vertices, we conclude that $T-v'$ is not a common card. 

As almost all vertices on the longest path of $T$ can not correspond to a common card, we will obtain an upper bound for $L$. This upper bound implies that almost all points on the unique cycle of $G$ must correspond to a common card. We combine this with the fact that $T$ is a tree to show that there are very few vertices on the cycle of $G$ such that the sum of the branches incident to it is maximal. So the gaps between these vertices will heavily increase as $n$ goes to infinity. Hence, we see that these branches will be in different places on different cards of $G$. On the other hand, these branches will be in (somewhat) the same places on different cards of $T$ as $T$ is a tree. This will yield a contradiction.

\begin{figure}[ht]
\centering
    \includegraphics[width = 6 cm]{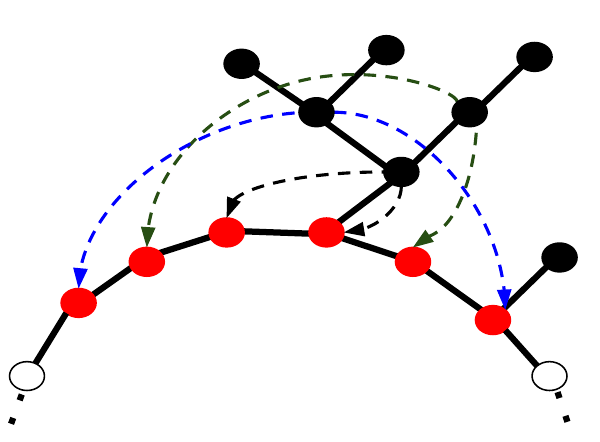}
    \caption{The coloring used in Lemma \ref{bound max path}. Every vertex of degree $\geq 2$ in a branch corresponds to two red vertices on the cycle that lie as close as possible to the root of the branch.}
     \label{Kleuring 1}
\end{figure}

\section{Proof of main theorem}\label{proof Thm}
Throughout this section $T$ will always be a tree and $G$ will always be a non-tree such that $T$ and $G$ have at least $\lfloor \frac{n}{2} \rfloor+2$ common cards. We will start by showing that $G$ is a connected unicyclic graph. Moreover, we use the convention that all vertices `$v$' lie in $T$, while all vertices `$w$' lie in $G$.

\subsection{First observations}
We will start by making some important observations and using counting arguments to prove some lemmas. According to a theorem by Myrvold \cite{WendyPhd}, the number of overlapping cards in the deck of a connected and a disconnected graph is at most $\lfloor \frac{n}{2} \rfloor +1$. So we may assume that $G$ is connected. We will now show that it suffices to only consider the case in which $G$ contains exactly one cycle as otherwise $G$ and $T$ have at most two common cards.

\begin{Lemma}\label{max one cycle}
A graph $G$ containing at least two cycles may have at most two common cards with a forest $F$.
\end{Lemma}
\begin{proof}
As the forest $F$ does not contain any cycles, neither do any common cards of $F$ and $G$. Next, we will prove that $G$ has at most two cards without any cycles. By contradiction, suppose that the cards $G-w_i$ for $i=1,2,3$ do not contain any cycles. Let $C$ and $C'$ be two different cycles in $G$. Because none of the cards $G-w_i$ contain any cycles, we immediately see that all $w_i$ lie both in $C$ and $C'$ as otherwise one of these cycles is visible on $G-w_i$. Since $C \neq C'$, and neither $G-w_1$ nor $G-w_2$ contains a cycle there must be at least three vertex disjoint paths between $w_1$ and $w_2$. The vertex $w_3$ can lie on at most one of these paths. Then the card $G-v_3$ contains two different paths between $w_1$ and $w_2$, implying that $G-w_3$ contains a closed walk and therefore $G-w_3$ contains a cycle. This gives a contradiction.
\end{proof}

Because of this lemma, we are only left with the case in which a tree $T$ and a connected unicyclic graph $G$ have at least $\lfloor\frac{n}{2}\rfloor+2$ common cards. As $T$ is a tree, it has $e(T)=n-1$ edges. Moreover, $e(G) = n = e(T)+1$ as $G$ is connected and contains exactly one cycle. Since $T$ does not contain a cycle, neither do any of the common cards. Hence, every vertex corresponding to a common card must lie in the cycle of $G$.
This immediately yields the following result:
\begin{Obs}
Let $L$ be the length of the unique cycle of $G$. Then $ L \geq \lfloor\frac{n}{2}\rfloor+2$.
\end{Obs}

Moreover, because every card contains at least $L-1$ vertices of this cycle and these vertices lie all in one path, we obtain the following: 

\begin{Obs} Every card of $G$ contains a path of length $L-1$.
\end{Obs}

Next, we will use a counting argument to show that there is a common card corresponding to a leaf in $T$.

\begin{Lemma}\label{card from leaf}
There is a common card $C = T-v = G-w$ such that $v$ has degree 1.
\end{Lemma}
\begin{proof}
Every common card $T-v=G-w$ satisfies $\deg(w)=\deg(v)+1$ as $e(G)=e(T)+1$. So if $\deg(v) \neq 1$ for all vertices $v$ of $T$ corresponding to a common card, then $\deg(w) \geq 3$ for all common card $G-w=T-v$. Hence, $G$ has at least $\lfloor\frac{n}{2}\rfloor+2$ vertices of degree at least $3$. Since $G$ is connected the degree of all the other vertices is at least 1. We get the following contradiction:  \[2e(G) \geq 3\left(\left\lfloor\frac{n}{2}\right\rfloor+2\right) + \left\lceil\frac{n}{2}\right\rceil - 2\geq 3\left(\frac{n}{2}+ 1\right)+\frac{n}{2}-2 = 2n+1 > 2n= 2e(G).\]
 \end{proof}

Next, we prove a lower bound for the number of leaves in $G$, which gives an upper bound for the number of vertices of degree at least $2$ outside the cycle of $G$. The expression $\Lf(X)$ denotes the number of leaves in a graph $X$.

\begin{Lemma}\label{min leaves}
 $G$ has at least $\frac{n}{4}-1$ leaves.
\end{Lemma}
\begin{proof}
We use proof by contradiction. If the statement is not true, then $G$ has at most $\frac{n}{4}-\frac{5}{4}$ vertices of degree 1, as the number of vertices of degree 1 is always an integer. Moreover, since $G$ is connected, $G$ does not contain any isolated vertices.

By Lemma \ref{card from leaf}, there exists a common card $C= G-w=T-v$ such that $v$ is a leaf in $T$, implying that $C$ is connected. Depending on whether $v$'s unique neighbor has degree 2, $C$  has either $\Lf(T)$ or $\Lf(T)-1$ leaves. Adding an extra vertex of degree $2$ to that card does not increase the number of leaves, but it can decrease the number of leaves by at most 2. Hence, $\Lf(T)-3 \leq \Lf(G)$ which implies that $\Lf(T) \leq \frac{n}{4}+\frac{7}{4}$.

Hence, at least $\frac{n}{4}-\frac{1}{4}$ of the common cards of $T$ correspond to a vertex of degree 2 or more. This indicates that $G$ has at least $\frac{n}{4}-\frac{1}{4}$ vertices of degree at least $3$ and the degree of all non-leaves is at least 2. Hence, we get the following contradiction:
\[2e(G) \geq 3\left(\frac{n}{4}-\frac{1}{4}\right)+ \left(\frac{n}{4}-\frac{5}{4}\right)+2\left(\frac{n}{2}+\frac{3}{2}\right) = 2n+1>2e(G).\]
 \end{proof}

\begin{Obs}
$G$ has at most $\frac{3n}{4}-L+1$ vertices of degree at least 2 outside its cycle.
\end{Obs}
\begin{proof}
There are exactly $n-L$ vertices outside the cycle, of which at least $\frac{n}{4}-1$ are leaves according to Lemma \ref{min leaves}. Thus, there are at most $n-L-\frac{n}{4}+1 = \frac{3n}{4}-L+1$ vertices of degree at least $2$ outside the cycle.
 \end{proof}

\subsubsection{Bounding the length of the cycle}
In this subsection, we will bound the length of the cycle from above. As mentioned in the proof overview, we begin by bounding the length of the longest path of $T$ in terms of the length $L$ of the cycle in $G$.

\begin{Lemma}\label{bound max path}
The longest path of $T$ has length at most $L+4$.
\end{Lemma}
\begin{proof}

 We fix a (planar) drawing of $G$. For every branch rooted at a vertex $w$ in the cycle in $G$ and that has exactly $k$ vertices of degree at least $2$ next to its root, we color the vertices on the cycle that lie within distance $k-1$ of $w$ red. Additionally, we color the vertex lying in the anticlockwise direction at a distance $k$ also red. We do this for every branch connected to the cycle. If we want to color an already red node, we do nothing. See Figure \ref{Kleuring 1 zonder pijlen} for a small sketch of this coloring. The advantage of this coloring is that it distinguishes between vertices that lie close to some large branch and the ones that lie `far enough away' from any branch. 

 \begin{figure}[ht]
     \centering
     \includegraphics[width = 8 cm]{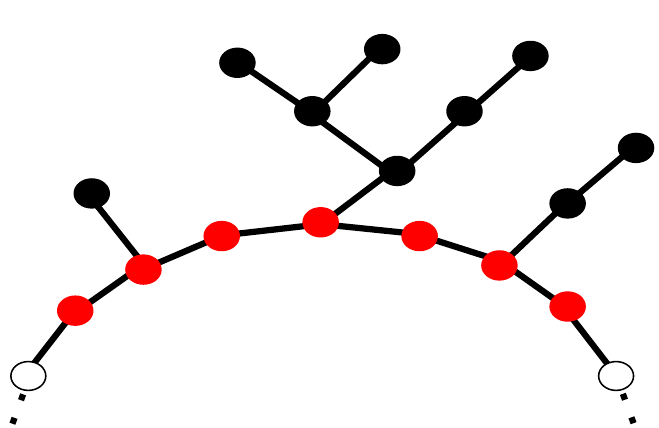}
     \caption{The coloring used in Lemma \ref{bound max path}. Every vertex of degree $\geq 2$ in a branch corresponds to at most two red vertices on the cycle that lie close to the root of the branch.}\label{Kleuring 1 zonder pijlen}
 \end{figure}
 
 Because the branches contain in total at most $\frac{3n}{4}+1-L$ vertices of degree at least $2$ and every such vertex corresponds to at most $2$ red vertices, there are at most $\frac{3n}{2}+2-2L$ red vertices in the cycle of $G$. Hence, there are at least $L-(\frac{3n}{2}+2-2L) \geq 3(L-\lfloor\frac{n}{2}\rfloor-2)+\frac{5}{2}$ white vertices on the cycle. Since there are exactly $L-\lfloor\frac{n}{2}\rfloor-2 \geq 0$ vertices on the cycle that do not correspond to a common card, there is at least one white vertex $w'$ corresponding to a common card $C=G-w'=T-v'$.

Now, we demonstrate that the card $C$ can only be in the deck of a tree if the longest path in that tree is at most $L+4$. Because no branch connected to $w'$ contains a vertex of degree at least $2$, the card $C$  consists of one large connected component and possibly also some isolated vertices. We color all vertices in $C$ that lie in the cycle of $G$ blue. Because $w'$ is white and every branch with a height of $k$ contains at least $k-1$ vertices of degree at least $2$, every such branch has a distance of at least $k-2$ from the right end of the blue path and of at least $k-1$ from the left end. Moreover, branches of height 1 extend the blue path by at most one at both ends. So the longest path is at most two longer than the blue path at the left end and at most one longer at the right end. Hence, the maximum path length on this card is at most $L-1+1+2=L+2$.

As $T$ is a tree, adding back $v'$ to $T$ only lets the largest component grow with $v'$ and the isolated vertices, which will become leaves. Therefore, the longest path will increase by at most two. Thus, the longest path in $T$ has a maximum length of $L+2+2=L+4$.
 \end{proof}

Since the longest path in $T$ is at most $L+4$, the length of every path on a common card is also at most $L+4$. We will use this to prove that there is some special common card $C$, such that the only vertices on a longest path $P$ corresponding to common cards are the vertices lying at one of the ends of $P$ (see Figure \ref{fig: goal Kleuring 2}). Since at most $\lfloor \frac{n}{2}\rfloor -2$ vertices of $G$ do not correspond to a common card, this observation leads to an upper bound for $L$.

\begin{Lemma}\label{bound cycle length}
The unique cycle in $G$ has a length of at most $\frac{n}{2}+19.5$.
\end{Lemma}

\begin{proof} Let $w$ be an arbitrary vertex on the cycle of $G$.
If the branch connected to $w$ contains exactly $k$ vertices of degree at least $2$, we color the vertex $w$ itself red, and we also color red all the vertices on the cycle that lie at a distance in $[11,k+9]$. We have chosen this interval because it contains exactly $k-1$ vertices and the higher upper bound of $k+9$ ensures that the longest path in the branch is at least 9 shorter than the path corresponding to the cycle. So if $k=1$, we initially only color the vertex $w$ red. Moreover, we also color the vertex at a distance 10 to the left red. Hence, for every vertex of degree at least $2$ in a branch, we color at most two vertices on the cycle red. Similar to the proof of the previous lemma, there is at least one common card $G-w'=T-v'$ corresponding to a white vertex $w'$ in $G$. Since $w'$ is white, the common card $G-w'=T-v'$ consists of one large component and maybe some isolated vertices.

\begin{figure}[ht]
    \centering
    \begin{subfigure}[b]{0.5\textwidth}
     \centering
    \includegraphics[width = 7 cm]{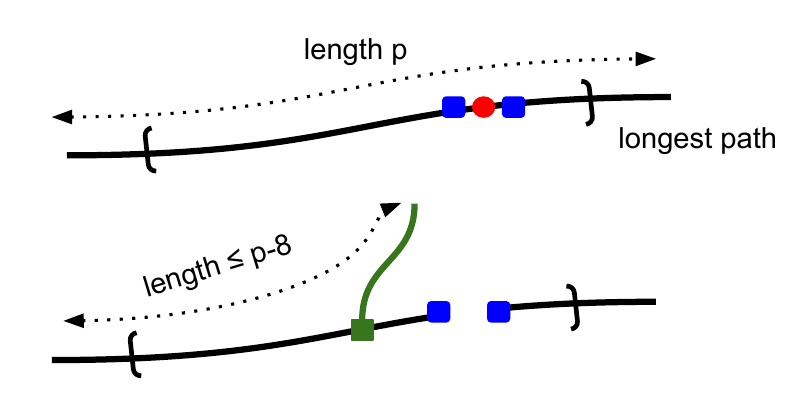}
    \caption{Find a vertex $w' \in V(G)$ such that not many vertices in $G-w' = T-v'$ can correspond to a common card. More specifically, we will show that the longest path in $T-v''$ will be too small for all vertices $v''$ in the middle of a longest path of $G-w' = T-v'$.}
    \label{fig: goal Kleuring 2}
     \end{subfigure}
     \hfill
     \begin{subfigure}[b]{0.48\textwidth}
         \centering
    \includegraphics[width = \textwidth]{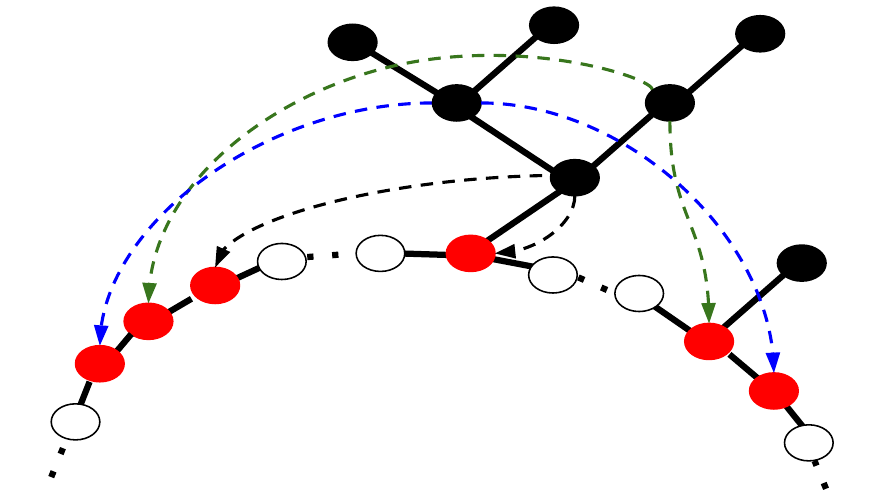}
    \caption{The coloring used in Lemma \ref{bound cycle length}. Every vertex of degree least $2$ corresponds to two red vertices on the cycle, but this time we have gaps of 9 and 10 white vertices between the vertex connected to the branch and the other red vertices.}
     \end{subfigure}
     \caption{The goal and a picture of the coloring in Lemma \ref{bound cycle length}.}\label{Kleuring2}
\end{figure}

Moreover, we color all points in $G-w'$ corresponding to a vertex on the cycle blue. For our argument, we do not need to know which vertices are colored blue, only that these vertices exist. Besides, let $P$ be a longest path of $G-w'=T-v'$.

First, we show that all except for maybe the 10 vertices on both ends of the blue path also lie in $P$. Suppose that this is not true. Because we have deleted a white vertex, every branch that lies at distance $d \geq 10$ of both ends of the blue path has a height of at most $d-8$, as every path of height $k$ in a branch contains at least $k-1$ vertices of degree at least $2$. So taking this branch results in a smaller path than following the blue path. Hence, these kinds of branches are not part of the longest path. Thus, $P$ contains all vertices on the blue path except for possibly the 10 vertices on one of the ends.

We will now show that all the vertices in the blue path, which are not in the last 10 vertices on either end, do not correspond to a common card. Let $v$ be a vertex on the blue path such that $v$ has a distance of at least 10 from both ends of the blue path. Our previous observation says that $v$ lies in the longest path $P$. Moreover, every path $Q$ in $T-v$ either stops at the vertex before $v$ or has taken a branch incident to a point before $v$.

First, we show that in $T-v'-v$, a path $Q$ that follows the blue path until it reaches the neighbor of $v$ is too small. Since both parts of $P$ looking from the vertex $v$ are as long as possible, the path $Q$ is smaller than or equal to the longest part of $P-v$. Moreover, both ends of $P$ from $v$ are at least as long as the respective ends of the blue path. Therefore, the path $P$ becomes at least 11 vertices shorter when we end the path $P$ at the vertex before $v$. Thus, $Q \leq P-11 \leq L+4-11=L-7$. By adding back the vertex $v'$ to $T-v'-v$, we can extend $Q$ only by $v'$ and some leaves (as $T$ is a tree and the other components only consist of isolated nodes). So, the extension of $Q$ in $T-v$ has length at most $L-5$.

Next, we show that every path $Q'$ in $T-v-v'$ that takes branch incident to a vertex $x$ such that $x$ has a distance of at least $10$ to both ends of the blue path in $T-v'$ is too small. By our choice of $w'$, the height of this branch is at most $d-8$, because otherwise, it would have at least $d-7-1=d-8$ vertices of degree at least $2$ implying that $w'$ was red.  This means that taking such a branch results in a path that is at least 8 vertices smaller than if we were to follow the blue path till the end. Moreover, as the longest path is at most 5 vertices longer than the blue path, we can only add at most 5 vertices in $G-w'$ to a path containing this branch by taking a branch near the other end of the blue path. This means that the length of such a path $Q'$ is at most $L-1-8+5=L-4$. As $T$ is a tree, adding back the vertex $v'$ makes the path $Q'$ at most 2 longer. So the extension of $Q'$ has length at most $L-2$ in $T-v$.

Combining these two steps, we find that the longest $T-v$ path has length at most $L-2$. Since every card of $G$ must contain a path of at least $L-1$, the card $T-v$ is not in the deck of $G$. This indicates that at most 20 vertices on the blue path in $T-v'$ can correspond to a common card. So at least $L-1-20=L-21$ vertices of $T$ do not correspond to a common card. Hence, $L-21 \leq \frac{n}{2}-1.5$, implying that $L \leq \frac{n}{2}+19.5$.
 \end{proof}

In particular, this result implies the following corollary.

\begin{Cor}\label{obs 18 vert}
     There are at most  18 vertices on the cycle of $G$ that do not correspond to a common card of $G$ and $T$.
\end{Cor}

Using this important corollary, we can bound the length of the branches:

\begin{Cor}\label{path branch}
The maximum height of a branch of $G$ is 14.
\end{Cor}
\begin{proof}
We use proof by contradiction. Suppose there is a vertex $w$ that has a branch with height at least 15. Given that the longest path on each common card is at most $L+4$, we can conclude that none of the vertices on the $G$ cycle with distance $d \in [1,10]$ of $w$ can correspond to a common card. For such vertices, we make a path that contains both the longest possible path in this branch and the largest part of the blue path. This new path has length at least $L-1+15 -(d-1) \geq L+5$, a contradiction. This indicates that at least $20$ vertices on the cycle do not correspond to a common card, which forms a contradiction with Corollary \ref{obs 18 vert}.
 \end{proof}

\begin{Cor}\label{path small component}
A path outside the largest component on a common card of $G-w=T-v$ has a length of at most 27.
\end{Cor}
\begin{proof} Small components in common cards only consist of vertices that are in one of the branches of $G$.
From the previous corollary, it follows that every branch has a height of at most 14. Therefore, every path consisting only of vertices in branches has a length of at most $2 \cdot 14-1=27$ as we only consider paths that do not contain the root, as they lie in the cycle of $G$.
 \end{proof}

\subsubsection{Proof of the main theorem}
Now, we have all the ingredients to prove the main theorem.

\maintrees*
\begin{proof}
First, we can determine whether a graph is connected from any $\lfloor \frac{n}{2}\rfloor+2$ of its cards. So we only need to examine the possibility that some tree $T$ and a connected unicyclic graph $G$ have at least $\lfloor \frac{n}{2}\rfloor+2$ cards in common.

Let $K$ be the maximum sum of the sizes of the branches incident to the same vertex in the cycle of $G$. Moreover, we color all vertices $w'$ for which this maximum is attained green. So a vertex $w$ of $G$ is green if and only if the largest component in $G-w$ has exactly $n-1-K$ vertices.

Lemma \ref{card from leaf} gives us a common card $T-v'=G-w'$ where $v'$ is a leaf and $w'$ is a vertex on the cycle of degree~$2$. Hence, this common card $C$ is connected.
We again color all vertices in $T-v'=G-w'$ that lie in the cycle of $G$ blue. If we add the leaf $v'$ back to $T-v'$, exactly one branch becomes larger. So there is at most one blue vertex for which the sum of the branches is at least $K+1$. This implies that there is at most one non-blue vertex in $T$ such that the largest component has exactly $n-1-K$ vertices.

If we delete vertices from the blue path, we observe that the largest component of $T$ becomes strictly smaller for every vertex we move closer to the center (until we reach the middle). Because we can approach the center from (only) two sides, both the endpoints of the blue path, there are at most 2 blue vertices such that the largest component has exactly $n-1-K$ vertices. In conclusion, there are at most 3 cards of $T$ such that the largest component has exactly $n-1-K$ vertices. 

This also implies that there are at most 3 such common cards. So at most 3 green vertices of $G$ can correspond to a common card.  Since there are at most 18 vertices in the cycle of $G$ that do not correspond to a card, there are at most $21$ green vertices in the cycle of $G$.

Let $n \geq 5000$, then $L \geq \frac{n}{2} +2> 2500$. Because there are at most 21 green vertices, the pigeonhole principle tells us that there is a `gap' between two green vertices consisting of at least $\frac{2500-21}{21} > 118$ adjacent non-green vertices on the cycle. Of these 118 vertices, at most 18 do not correspond to a common card.

Let $w'$ be one of the vertices in this gap of $118$ vertices such that there are at least 29 vertices on both sides of $w'$ in this gap.  There are at least $118-2 \cdot 29 = 60$ such vertices $w'$ and at least $60-18=42$ of them correspond to a common card. We will show that the green vertices in $G-w'=T-v'$ are exactly those vertices in the large component that split the large component into two components that have a path of length (at least) $29$ and exactly $K$ other vertices. 

Since $w'$ lies on the blue path and the distance between $w'$ and any green vertex is at least $29$, the statement holds for all green vertices. Let $w''$ be a vertex that satisfies this condition, we will show that $w''$ is green. By Corollary \ref{path small component} all paths consisting only of vertices in the branches have a length of at most 27. This means that both components with a path of length $29$ must contain some blue vertices. Hence, $w''$ lies on the blue path and thus on the cycle of $G$. Moreover, all the extra vertices in components without a path of length at least 29, were part of a branch of $w''$. As the green vertices are the only ones with (at least) $K$ vertices in the adjacent branches we conclude that $w''$ must have been green. Therefore, we can identify the green vertices in $G-w'=T-v'$.

Because the maximum path in a branch has size at most 27, we know that $T$ has no `fake green' vertices that do not lie on the main path i.e. $T$ has no vertices in a branch such that the largest component splits into two components of size at least 29 and $K$ extra vertices. Hence, in $T$ there is a path through all green vertices. Let $D_1$ and $D_2$ (with $D_1 \leq D_2$) be the maximum number of vertices before the first green vertex and after the last green vertex in a path through all green vertices.

For every vertex $w'$ in the gap of distance at least 30, we define $e_1^{w'}$ and $e_2^{w'}$ (such that $e_1^{w'} \leq e_w^{w'}$) to be the number of vertices between $w'$ and two green vertices at the boundary of the gap. By definition,  $e_1^{w'}+e_2^{w'}$ equals the size of the gap minus one. Since all the green vertices lie on the cycle of $G$, there is a longest path $P$ in $G-w'=T-v'$ through all the green vertices. Then we define $d_1^{w'}$ and $d_2^{w'}$ (with $d_1^{w'} \leq d_2^{w'}$) to be the maximum number of vertices before the first green vertex or after the last green vertices for a path through all green vertices in $G-w'=T-v'$.

As the blue path goes through all the green nodes and the longest path in $T$ is at most 5 longer than the blue path, we see $d_1^{w'1}+d_2^{w'} \leq e_1^{w'}+e_2^{w'}+5$ and $e_i^{w'} \leq d_i^{w'} \leq e_i^{w'}+5$ for $i=1,2$. We can define these numbers for all $60$ possible $w'$. 

If we add the missing vertex $v'$ in $T-v'$ back at most one of $d_1^{w'}$ and $d_2^{w'}$ can change. Since the updated values of $d_1^{w'}$ and $d_2^{w'}$ must be equal to the `true' values $D_1,D_2$ in $T$, we see that $d_i^{w'}= D_j$ for some $i,j \in \{1,2\}$. Hence, 
\[e_i^{w'} \leq D_j \leq e_i^{w'}+5.\]
Moreover, $e_1^{w'}+e_2^{w'} \leq D_1+D_2 \leq e_1^{w'}+e_2^{w'}+5$. Combining these results gives $D_{3-j} - 5 \leq  e_{3-i}^{w'} \leq D_{3-j}+5$. Because $D_1 \leq D_2$ and $e_1 \leq e_2$, we get that $e_1^{w'} \in [D_1-5,D_1+5]$. Hence, the minimum distance to a green vertex differs by at most 10 
over all $w'$. 

However, for every integer $n$ in $[30,83]$, there are exactly two vertices $w'$ in the gap with minimal distance $n$ to a green vertex. Since there are at most 18 vertices that do not correspond to a card, there is such a vertex $w'$ corresponding to a common card with a minimal distance to a green vertex in the interval $[30,39]$ and one vertex $w''$ in the interval $[50,59]$. These two vertices have a difference in minimal distance of at least $50-39=11>10$, a contradiction.
 \end{proof}

 \begin{Remark}
     We can slightly improve on this upper bound by considering a tree $T$ and a graph $G$ with $\frac{n}{2}-k$ common cards. For these graphs, we can do (roughly) the same proof, to get a contradiction. The main difference is that we apply the colorings only for $\ell-(2k+2)$ vertices for every branch of size $\ell$. Then we get a contradiction, whenever $n> 6 000k^2 +24 000k+ 24 000$. In particular, for $k=\frac{\sqrt{n}}{80}$ we get a contradiction as long as $n>70 000$.
 \end{Remark}

\section{Adversary recognition for forests}\label{chapter: forest}
In the previous section, we have shown that, for $n$ large enough, any $\lfloor\frac{n}{2}\rfloor+2$ cards determine whether a graph $G$ is a tree. Next, we want to investigate how many cards we need to determine whether a graph is a forest i.e. whether the graph contains a cycle. 

We start by proving a lower bound of $\lfloor\frac{n}{2}\rfloor+1$ by giving three families of pairs of a forest $F$ and an unicyclic graph $G$ that have $\lfloor\frac{n}{2}\rfloor+1$ common cards. In order to state these families, we use the following notation: $C_k$ stands for the cycle on $k$ vertices, $P_k$ for the path on $k$ vertices, $K_2$ for an isolated edge, and $K_1$ for an isolated vertex. Moreover, we need the following definition.
\begin{Def}
    Let $G$ be any graph, then the $\mathrm{star \, of } \, G$, denoted $S_1[G]$, is the graph $G$ with an extra leaf connected to each of its vertices. If $G$ is vertex-transitive, then $S_1[G]'$ is the graph $S_1[G]$ minus any leaf.
\end{Def}
The requirement in the second part that the graph $G$ must be vertex-regular is necessary as otherwise the graph $S_1[G]'$ is not well-defined. 

\begin{Obs}\label{three families}
There are (at least) three infinite families of pairs consisting of a forest $F$ and a unicyclic graph $G$ that have $\lfloor \frac{n}{2}\rfloor+1$ common cards.
\begin{itemize}
    \item Family 1: $F_{2k+1}= P_{k} \cup  (k+1)K_1$ and $G_{2k+1} = C_{k+1} \cup k K_1$ (see Figure \ref{family 1})
    \item Family 2: $F_{4k-1}=P_{2k-1}\cup k K_2$ and $G_{4k-1}=C_{2k} \cup (k-1)  K_2 \cup K_1$ (see Figure \ref{family 2}).
    \item Family 3: $F_{2k+1}=S_1[P_k] \cup K_1$ and $G_{2k+1}= S_1[C_{k+1}]'$ (see Figure \ref{fig:forest2}). 
\end{itemize}
\end{Obs}

\begin{figure}[ht]
\centering
\begin{subfigure}[b]{0.5 \textwidth}
    \centering
    \includegraphics[width = 6.5 cm]{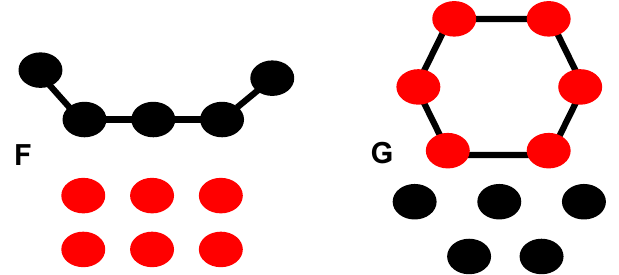}
    \caption{An example of the first family ($k=5$): The forest $F$ is a path with some isolated vertices and a unicyclic graph $G$ is a cycle with some isolated vertices.}\label{family 1}
\end{subfigure}
\hfill
\begin{subfigure}[b]{0.48 \textwidth}
    \centering
    \includegraphics[width = 6.5 cm]{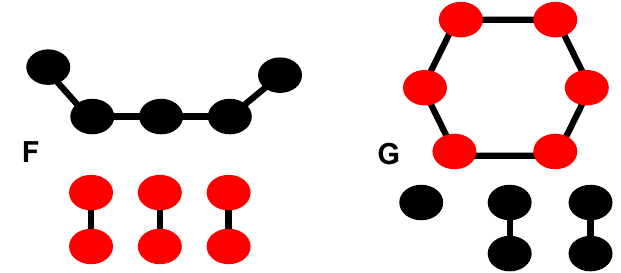}
    \caption{An example of the first family ($k=3$): The forest $F$ is a path with some isolated edges and a unicyclic graph $G$ is a cycle with some isolated edges and one isolated vertex.}\label{family 2}
\end{subfigure}
\caption{All red vertices correspond to common cards of $F$ and $G$, that are all isomorphic.}\label{fig:forest1}
\end{figure}

\begin{figure}[ht]
    \centering
    \includegraphics[width = 10 cm]{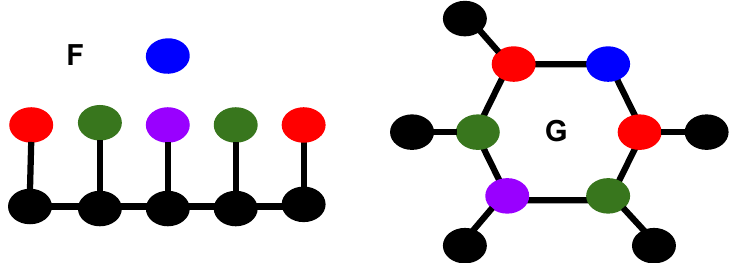}
    \caption{An example corresponding to Family 3 ($k=5$): The forest $F$ is the star of $P_5$ with an isolated vertex, while $G$ is the star of $C_6$ minus one leaf. All vertices of the same color (not black) correspond to the same cards.}\label{fig:forest2}
\end{figure}

\begin{Remark} These families were already known examples of pairs that have $\lfloor \frac{n}{2}\rfloor+1$ common cards.  For example, Family 3 is a special case of Theorem 3.4 from `Families of Pairs of Graphs
with a Large Number of
Common Cards' \cite{bowler2010families}. \end{Remark}

To obtain the common cards of $F$ and $G$, we always must delete a vertex from the unique cycle of size $\lfloor \frac{n}{2} \rfloor+1$ of $G$ as common cards do not contain any cycles. Moreover, they correspond to vertices outside the middle $\lfloor \frac{n}{2} \rfloor$ vertices of the longest path of $F$ as the longest path on each card is at least $\lfloor \frac{n}{2}\rfloor$. Figures \ref{fig:forest1} and \ref{fig:forest2} illustrate this process. 

This means that we have found a lower bound of $\lfloor \frac{n}{2}\rfloor+2$ for determining whether a graph is a forest. In this next section, we prove that, for $n$ large enough, the lower bound is in fact strict.

\subsection{Proof of the upper bound}
For the proof of the upper bound, we will use a case distinction based on the following properties of forest $F$ and unicyclic graph $G$: the number of connected components (denoted by respectively $\kappa_F$ and $\kappa_G$), and the size of the largest connected component (denoted by respectively $M_F$ and $M_G$). We start by proving a lemma that gives us a relation between $\kappa_F$ and $\kappa_G$.

\begin{Lemma}\label{relation K_F}
If a forest $F$ and a non-forest $G$ have at least $\lfloor \frac{n}{2} \rfloor+1$ common cards, then $\kappa_F \leq \kappa_G+1$.
\end{Lemma}
\begin{proof}
As the graph $G$ is not a forest, $G$ contains a cycle. Lemma \ref{max one cycle} tells us that $G$ has exactly one cycle. As common cards do not contain cycles, every vertex corresponding to a common card must lie in the unique cycle because this cycle and thus the length $L$ of this cycle satisfies $L \geq \lfloor \frac{n}{2} \rfloor+1$.

As $G$ is unicyclic, there are at least $L-(n-L) \geq 2(L-\lfloor\frac{n}{2}\rfloor - 1) +1$ vertices of degree 2 in the cycle. Since at most $L-\lfloor\frac{n}{2}\rfloor - 1$ vertices on the cycle do not correspond to a common card, there is a common card $G-w=F-v$ such that $\deg(w)=2$.

As $G-w=F-v$ and subgraphs of forests are forests, the card $G-w$ is also a forest. Moreover, the connected components of $G-w$ are the same as $G$ except for the one with the cycle (which is still one connected component as we deleted a node of degree 2). In particular, $\kappa_{G-w} = \kappa_G$. Because $G-w=F-v$, we also have that $\kappa_{F-v}=\kappa_{G-w}$. At last, $\kappa_{F} \leq \kappa_{F-v}+1= \kappa_G + 1$ as a card $F-v$ only contains fewer components than $F$ if $v$ is an isolated vertex and in that case, the number of components decreases by exactly one. 
\end{proof}

To prove the main theorem, we will use a case distinction based on the difference between $\kappa_F$ and $\kappa_G$. For two of these cases, we will need to use subcases which are determined by $M_F$ and $M_G$, the sizes of the largest component of $F$ and $G$.
\mainforest*
\begin{proof}
We will use proof by contradiction. Let $F$ be a forest and $G$ be a non-forest that have at least $\lfloor\frac{n}{2}\rfloor+2$ common cards. Then by Lemma \ref{max one cycle} $G$ has exactly one cycle and we again have that every common card must correspond to a vertex in the cycle of $G$. Hence, this cycle has length at least $\lfloor\frac{n}{2}\rfloor+2$ implying that the largest component of $G$ is at least two larger than the next largest component.

By Lemma \ref{relation K_F}, the graphs $F$ and $G$ satisfy $\kappa_F \leq \kappa_G+1$. Moreover, $e(F)=n-\kappa_F$ as $F$ is a forest and $e(G)=n+1-\kappa_{G}$ because $G$ contains exactly one cycle. Furthermore, the proof of Lemma \ref{relation K_F} tells us that there exists a common card $G-w=F-v$ such that $w$ has degree 2. This cards satisfies $M_{F} \geq M_{F-v}=M_{G-w} = M_G-1$ implying $M_F \geq M_G-1$.

For the remainder of the proof, we will distinguish four cases, where the second and third case are each split into two subcases. We will show that for every case there are no forests $F$ and unicyclic graphs $G$ that have $\lfloor \frac{n}{2}\rfloor+2$ common cards.\\

Case 1: $\kappa_F=\kappa_G+1$ implying  $e(F)=n-\kappa_F=n-\kappa_{G}-1=e(G)-2$.\\
In this case, every common card that corresponds to a vertex of degree $2$ in $G$ corresponds to an isolated vertex in $F$. Let $C$ be the number of common cards of this form. Then $F$ has at least $C$ isolated vertices implying that these common cards have at least $C-1$ isolated vertices. Since deleting a vertex of degree 2 from the cycle of $G$ does not increase the number of isolated vertices, $G$ has also at least $C-1$ isolated vertices.
  
Furthermore, for every vertex of degree $d \geq 3$ in the cycle of $G$, there is at least one extra vertex in the largest component outside the cycle connected to that vertex. Since there are $\lfloor \frac{n}{2}\rfloor+2-C$ common cards corresponding to a node of degree $d \geq 3$, $G$ has at least $\lfloor \frac{n}{2}\rfloor +2 + (\lfloor \frac{n}{2}\rfloor +2-C)+C-1\geq n+2$ vertices, a contradiction. \\

Case 2:  $\kappa_F= \kappa_G$ implying $e(F)=e(G)-1$.\\
For a common card $G-w=F-v$ where $w$ has degree 2, the missing vertex $v$ has degree 1. Thus, $M_F=M_{F-v}+1$ or $M_F=M_{F-v}$, implying $M_F=M_G$ or $M_F=M_G-1$.\\
Case 2a: $\kappa_F= \kappa_G$  and $M_F=M_G-1$.\\
In this case, every common card corresponding to a vertex of degree 2 in $G$, corresponds to a vertex of degree 1 outside the main component in $F$. So if $C$ is the number of common cards corresponding to a vertex of degree 2, $F$ has at least $C$ vertices outside its main component. As $M_F=M_G-1$, $G$ has at least $C-1$ vertices outside its main component. Moreover, for every vertex of degree at least 3 on the cycle of $G$, there is at least one vertex in the main component of $G$ outside the cycle. This means that $G$ has at least
$\lfloor \frac{n}{2}\rfloor +2 + (\lfloor \frac{n}{2}\rfloor +2-C)$ vertices in its main component. Hence, $G$ must have more than $n$ vertices, a contradiction.\\

Case 2b: $\kappa_F= \kappa_G$  and $M_F=M_G$.\\
As we always delete a vertex from the largest component of $G$ and $M_F=M_G$, we also always delete a vertex from the largest component of $F$. There is at least one common card corresponding to a vertex of degree $2$ in $G$ and to a leaf of the main component in $F$. As deleting this vertex in $G$ does not create any extra small components (and neither does deleting any vertex of degree $1$ from the largest component of $F$),  all small components of $F$ and $G$ must be isomorphic. As we only delete nodes from the largest components of $F$ and $G$, these small components are always visible on common cards. Hence, the largest component of $F$ (which is a tree) and the largest component of $G$ (which is a connected unicyclic graph) have at least $\lfloor \frac{n}{2}\rfloor+2$ common cards, while their size is at most $n$. This forms a contradiction with Theorem \ref{main thm trees}.\\

Case 3: $\kappa_F=\kappa_G-1$ implying $e(F)=e(G)$.\\
Case 3a: $\kappa_F=\kappa_G-1$ and $M_F=M_G-1$.\\
In this case, every common card corresponding to a vertex of degree 2 in $G$ corresponds to a vertex of degree 2 outside the large component in $F$. Because there is at least one common card corresponding to a vertex of degree $2$ in $G$, this means that $F$ has at least one vertex of degree 2 outside its main component. However,  $F$ is a forest, so this also means that $F$ has at least 2 extra vertices of degree $1$ outside its main component.

On the other hand, for every common card corresponding to a vertex of degree $3$ in $G$, there must be at least one new vertex of degree 1 outside the cycle connected to that vertex. With this bijection, we can show, in a similar way as in case 2a, that $G$ must have more than $n$ vertices to derive a contradiction.\\

Case 3b: $\kappa_F=\kappa_G-1$  and $M_F > M_G-1$.\\
As every common card $G-w=F-v$ corresponds to a vertex $w$ in the largest component of $G$ and $M_F > M_G-1$, the vertex $v$ must lie in the largest component of $F$. Moreover, since $w$, and therefore also $v$, always have degree at least 2, we have $M_{F} \geq M_{F-v}+2 = M_{G-w}+2=M_{G}+1$. Therefore, there exists a connected graph $H$ with at most $\lceil\frac{n}{2}\rceil-2$ vertices such that $G$ contains more small connected components of $H$ than $F$. 

As all cards correspond to a vertex in the main component of $G$, all the copies of $H$ are visible on every common card. As every common card corresponds to a vertex in the largest component $M_F$ from $F$, this largest component $M_F$ of $F$ has at least $\lfloor \frac{n}{2}\rfloor+2$ cards that contain a component isomorphic to $H$. This forms a contradiction with Lemma \ref{component H in connected} from the Appendix.\\

Case 4: $\kappa_F < \kappa_G-1$, so $e(F)>e(G)$. \\
As all cards of $G$ correspond to a vertex of degree $
\geq 2$, the forest $F$ has at least $\lfloor \frac{n}{2}\rfloor +2$ vertices of degree $\geq 3$. Since every connected component of $F$ is a tree and a tree with $k$ vertices of degree $\geq 3$ has at least $k+2$ leaves, $F$ has at least $\lfloor\frac{n}{2} \rfloor +4$ vertices of degree leaves. This means that $F$ must have more than $n$ vertices, a contradiction.
\end{proof}

This concludes the proof that  forests and non-forests have at most $\lfloor \frac{n}{2}\rfloor+1$ common cards. Moreover, this bound is strict since we already have found three infinite families of pairs of forests and non-forests for which this bound is attained.  

\begin{Remark} The requirement that $n \geq 5000$ is only necessary in case 2b where we use Theorem \ref{main thm trees}. This means that if Theorem \ref{main thm trees} is proven for all $n \geq k$ (where $k$ is some positive integer), then also Theorem \ref{main forest} is proven for all $n \geq k$.
\end{Remark}

\section{Adversary recognition for the girth}

The next graph property that we will examine is the girth, which is the length of the smallest cycle in a graph. For example, the girth of $C_n$ is $n$ and the girth of $K_n$ is $3$ for $n\geq 3$. If the graph does not contain any cycle, the girth is defined to be  infinity. The three infinite families of pairs from the previous section show that - in the worst case - we may need at least $\lfloor \frac{n}{2}\rfloor+2$ cards to determine the girth. This is the case as the girth for forests is infinite, while it is finite  for unicyclic graphs. Hence, the graphs in these pairs do not have the same girth. In this section, we will show that the girth of a graph $G$ can be determined from  any of its $\frac{2n}{3}+1$ cards. We start with an elementary lemma about graphs where the girth is at least $\frac{2n}{3}+1$. A full proof of this result can be found in the Appendix.

\begin{Lemma}\label{max een cycle}
   Every graph $G$ on $n$ vertices with girth at least $\frac{2n}{3}+1$ has  at most one cycle.
\end{Lemma}
\begin{proof}
As two disjoint cycles will be too small, the worst case is a graph $G$ that consists of $2$ vertices with $3$ paths between them. Then we have three cycles that are obtained from taking two of these three paths. Hence, the length of the smallest cycle is at most $2+\frac{2(n-2)}{3} < \frac{2n}{3}+1$.
\end{proof}

In chapter \ref{chapter: forest}, we saw that for $n \geq 5000$ a forest and a non-forest on $n$ vertices have at most $\lfloor \frac{n}{2}\rfloor+1$ common cards. However one can easily prove, even for small $n$, using (mostly) counting arguments that 
the following result holds:

\begin{restatable}{Lemma}{forestvscycle}\label{forest vs cycle}
 Forests and non-forests have at most $\frac{2n}{3}$ common cards.
\end{restatable}

We will now combine the previous two lemmas with a version of the previous coloring  techniques to prove the following theorem.

\thmgirth*
\begin{proof}
First, from Lemma \ref{forest vs cycle} we can determine whether a graph has a cycle based on any $\frac{2n}{3}+1$ of its cards. So if $G$ does not have any cycle, we know that its girth is $\infty$.

Let $G$ be an arbitrary graph with a cycle and let $L$ be the girth of $G$. Any cycle of length at most $\frac{2n}{3}$ appears on at least one of the $\frac{2n}{3}+1$ cards. This implies that if the girth $L$ of $G$ is $\leq  \frac{2n}{3}$, we see a cycle of length $L$ and since we do not see any smaller cycles, we know that $L$ is the girth of $G$. Hence, we will only consider graphs $G$ with at least one cycle and a girth of $L \geq \frac{2n}{3}+1$. 

By Lemma \ref{max een cycle}, these graphs $G$ have exactly one cycle. If this unique cycle is visible on at least one card, we can use this card to determine the girth $L$. Hence, we assume that this cycle is not visible so that every card in the subdeck corresponds to a vertex in the cycle of $G$. These cards $G-v$ all contain a path of length $L-1$ because $G$ contains a cycle of length $L$.

To determine the length of the unique cycle, we will use a coloring somewhat similar to the one in Figure \ref{Kleuring 1} to distinguish good and bad vertices on the cycle. If there are $k$ vertices in a branch connected to a vertex $v$ on the cycle, then we color all vertices on the cycle with a distance between 1 and $k$ of $v$ red. In particular, we do not color the vertex $v$ itself red. So every vertex outside the cycle corresponds to at most two  red vertices. Thus, there are at least
\[L-2(N-L) = 3L-2N= 3\left(L-\frac{2n}{3}-1\right)+3\]
white vertices on the cycle. Since at most $L-\frac{2n}{3}-1$ vertices on the cycle do not correspond to a card in the subdeck, there is at least one card in the subdeck that corresponds to a white vertex $w$.

Every branch with a path of length $k$ from the cycle contains at least $k$ vertices. Therefore, such branches are either incident to $w$ or have distance at least $k+1$ from $w$. So on the card $G-w$, these branches lie in such places that by taking them instead of the cycle does not result in a longer path. This means that the longest path on $G-w$ has a length of exactly $L-1$.

In conclusion, we can reconstruct $L-1$ because it is the smallest among the longest paths over all cards. So we can determine $L-1$ and thus also the girth $L$ self. 
\end{proof}

\section{Adversary recognition for bipartiteness}
The last property that we will examine is bipartiteness. A graph is bipartite if it does not contain any cycle of odd length.  The main goal of this chapter is to show that we can determine whether a graph $G$ has such a cycle even if there is a linear amount of cards missing from its deck.

\thmbipartite*

We will use a similar approach as in the proof of Theorem \ref{main thm trees}. In this case, we consider a bipartite graph $G$ and a non-bipartite graph $H$ that have at least  $\frac{5n}{6}+2$ common cards. We fix $D$ to be a subdeck of $\frac{5n}{6}+2$ of these common cards. All observations and results in this section refer to this $G,$ $H$ and $D$ under these conditions.

We will call cycles of length at most  $\frac{5n}{6}+1$ \textit{small cycles} 
 because these cycles are all visible on at least one card of $D$. All cycles with at least $\frac{5n}{6}+2$ are called \textit{large cycles}. This gives us the following observations.

\begin{Obs}\label{small cycles}
    Let $k \leq \frac{5n}{6}+1$. Then $G$ contains a cycle of length $k$ if and only if $H$ does.
\end{Obs}

As $G$ does not contain an odd cycle, neither do cards in $D$. Thus all small cycles of $G$ and $H$ have even length.

\begin{Obs}
All odd cycles in $H$ are large cycles, so all small cycles of $H$ are even.
\end{Obs}

For the rest of this section, we fix a smallest odd cycle $L$ of $H$. The previous observation gives $L \geq \frac{5n}{6}+2$, where we abuse notation by using $L$ for the number of vertices in $L$. We will continue this notation convention in the rest of the proof. As the cycle $L$ can not be visible on any common card, all vertices corresponding to a common card $G-v=H-w$ lie in $L$. We will use this fact to  prove that many common cards in $D$ correspond to vertices of degree 2. With this information, we will derive an upper bound for the largest even cycle of $H$.

\begin{Lemma}\label{lotsdeg2} At least $\frac{2n}{3}+4$ cards of $D$ correspond to a vertex of degree 2 in $H$.    
\end{Lemma}
\begin{proof}
    First, note that there are exactly $n-L$ vertices outside the cycle $L$ of $H$. Now we will prove that every vertex outside the cycle $L$ can be connected to at most $2$ vertices on the cycle $L$. First, notice that if $v$ is connected to $w_1$ and $w_2$, the distance between $w_1$ and $w_2$ should be even because otherwise we get an odd cycle of length at most $L/2+1 \leq \frac{n}{2}+1$. Hence, there is at least one vertex on the cycle of $L$ between $w_1$ and $w_2$. Moreover, we can replace the (short even) arc $w_1w_2$ of $L$ with the path $w_1vw_2$ to get a new odd cycle. Since this cycle is at least as big as $L$,  there is at most one vertex between $w_1$ and $w_2$. Thus, there is always exactly one vertex  between $w_1$ and $w_2$ in $L$, which concludes the proof of the claim.
      
    Next, we see that if a vertex $v$ outside the cycle $L$ is connected to two vertices $w_1$ and $w_2$ on the cycle, that then the vertex on the cycle between $w_1$ and $w_2$, called $w$, can not correspond to a common card, as the card $H-w$ contains a cycle of length $L$ consisting of the long arc $w_2w_1$ and the `branch' $w_1 v w_2$. This means that every vertex on the cycle that is connected to two vertices on $L$ corresponds to a unique vertex on the cycle that can not correspond to a common card. Since this vertex lies directly between the two vertices whose degree increased, we see that every vertex $v$ which increases the degree of two vertices on the cycle of $G$ from  2 to 3, adds a new vertex to the cycle that can not correspond to a common card.

    Hence, we see that on average every vertex that does not correspond to a card increases the degree of at most one vertex corresponding to a card from $2$ to something higher. Thus, there are at least $\frac{5n}{6}+2-(n-\frac{5n}{6}-2) \geq \frac{2n}{3}+4$ cards in $D$ corresponding to a vertex of degree 2.
\end{proof}

The following shows that the number of edges of $G$ and $H$ are equal.

\begin{Lemma} $e(G)=e(H)$. \label{same number of edges} \end{Lemma}
\begin{proof}
Suppose $e(G) \neq e(H)$. The at least $\frac{2n}{3}+4$ cards that correspond to a vertex of degree 2 in $H$, correspond to a vertex of degree $d$ in $G$ implying that $G$ has at least $\frac{2n}{3}+4$ vertices of degree $d$. Moreover, by deleting a vertex of degree $2$ from $H$ we get a common card with exactly $\frac{2n}{3}+1$ vertices of degree 2. 

So if $d>3$, then we can consider this common card to show that $G$ has at least $\frac{2n}{3}+1$ vertices of degree at most 3. But then $G$ would have at least $\frac{4n}{3}+4$ vertices, a contradiction. If $d=1$, then $G$ this common card gives that $G$ has at least $\frac{2n}{3}$ vertices of degree $2$, implying that $G$ has again too many vertices. And similarly, if $d=3$ $G$ has at least $\frac{2n}{3}-2=\frac{2n}{3}$ vertices of degree at most 2, which also leads to a contradiction. Hence, $d=2$ implying $e(G)=e(H)$. \end{proof}

We call a vertex $w$ in $L$ \textit{simple} if $\deg(w)=2$ and if $w$ does not lie in any even cycle $C$.

\begin{Lemma}\label{simplevertices}
    Every vertex of degree 2 of $H$ corresponding to a common card is simple.
\end{Lemma}
\begin{proof}
    Suppose $w$ is a vertex of degree 2 in $H$ corresponding to a common card, but $w$ is not simple. Then $w$ lies in an even cycle $C$ and in $L$.  Let $w_1$ and $w_2$ be the neighbors of $w$.  Since $w$ has degree 2, the vertices $w_1$ and $w_2$ also lie in $C \cap L$. Replacing the path $w_1ww_2$ in $L$ with the other arc of $C$ gives us an odd closed walk. This odd closed walk lies also in $H-w$ implying that $H-w$ also contains an odd cycle, a contradiction.
\end{proof}

As $H$ has at least $\frac{2n}{3}+4$ vertices of degree $2$ that correspond to a common card, we can conclude that:

\begin{Cor}\label{max length even cycle}
    The largest even cycle of $H$ has length at most $\frac{n}{3}-4$.
\end{Cor}
\begin{proof}
    Combining Lemmas \ref{lotsdeg2} and \ref{simplevertices} results that $H$ has at least $\frac{2n}{3}+4$ vertices that do not lie in any even cycle. Hence, every even cycle has size at most $n-\frac{2n}{3}+4= \frac{n}{3}-4$.
\end{proof}

If we combine this lemma with Observation \ref{small cycles}, we get the following corollary:

\begin{Cor}
The largest small cycle in $G$ has size at most $\frac{n}{3}-4$.
\end{Cor}

With this information, we can show that common cards contain many cut vertices.

\begin{Cor}\label{cutvertices}
    Every common card $G-v= H-w$ contains $\frac{2n}{3}+1$ cut vertices of degree 2.
\end{Cor}
\begin{proof}
    $H$ has at least $\frac{2n}{3}+4$ simple vertices, which by definition are of degree 2 and do not lie in any even cycle. Since common cards do not contain odd cycles, none of these vertices lie in a cycle on a common card. By deleting a vertex at most one vertex of degree 2 disappears and the degree of at most two vertices of degree 2 in the cycle $L$ decreases. This means that common cards have at least  $\frac{2n}{3}+4-1-2$ cut vertices of degree 2.
\end{proof}

We now show that $G$ also has a large cycle.
\begin{Lemma}
$G$ contains a cycle of size at least $\frac{5n}{6}+2$.
\end{Lemma}
\begin{proof}
    Suppose not. Then Lemmas \ref{small cycles} and \ref{max length even cycle} imply that the maximum length of any cycle in $G$ is at most $\frac{n}{3}-4$. Let $G-w=H-v$ be a common card corresponding to vertices of degree 2. By Corollary \ref{cutvertices} this card has at least $\frac{2n}{3}+1$ cut vertices of degree 2. However, by adding back the node $w$ of degree 2, at most $\frac{n}{3}-5$ cut vertices will be put in a cycle as every cycle of $G$ contains at most $\frac{n}{3}-4$ vertices. This means that $G$ has at least $\frac{2n}{3}+1-(\frac{n}{3}-5)= \frac{n}{3}+6$ cut vertices of degree 2. In particular, at least one of these cut vertices $v$ corresponds to a common card $G-v=H-w$. This card is disconnected as $v$ was a cut vertex. However, for every common card $G-v=H-w$, where $\deg(w)=2$, is connected as $w$ must lie in $L$. Hence, the vertex $v$ does not correspond to a common card, a contradiction.
\end{proof}

Let $K$ be the smallest large cycle of $G$.  $G$ is bipartite, so the cycle $K$ has even length. Since the even cycles in $H$ contain at most $\frac{n}{3}-4$ vertices, this cycle is not visible on any common card, implying that all vertices $v \in G$ corresponding to a common card lie in $K$. Let $C$ be an arbitrary small cycle in $G$ and/or $H$.
\begin{Lemma}\label{distance}
    For every two vertices $u,v \in K \cap C$ (respectively $L \cap C$) the distance in $K$ (or $L$) is smaller or equal to the distance in $C$.
\end{Lemma}
\begin{proof}
    We only prove the statement for $K$ as the proof for $L$ is the same. If we replace the short arc of $K$ with the short arc in $C$, the cycle should either become larger (implying that the short arc in $C$ is larger than the short arc in $K$) or smaller. 
    However, if it becomes smaller we must obtain a small cycle by definition of $K$. This means that the small arc in $C$ is at least $\frac{5n}{6}+2-(\frac{n}{3}-4) \geq \frac{n}{2}+6$ smaller than the small arc in $K$. However, since the length of $K$ is at most $n$, the small arc has size at most $\frac{n}{2}$, a contradiction.
\end{proof}

\begin{Cor}\label{interval}
    For every small cycle $C$, the vertices of $K \cap C$ (respectively $L \cap C$) lie in one interval of size at most $\frac{C}{2}+1$ in $K$ (respectively $L$).
\end{Cor}
\begin{proof}
     We only prove the statement for $K$ as the proof for $L$ is the same. Corollary \ref{distance} gives that every two vertices in $C \cap K$ lie within distance $\frac{C}{2}$ in $K$. Fix a vertex $v \in K \cap C$, then every vertex lies within distance $\frac{C}{2}$ from a vertex $v$. Since $2C \leq \frac{2n}{3}-8 < K$, the shortest arc between vertices left and right from $v$ is always via $v$. Hence, all the vertices in $K \cap C$ lie in an interval of size at most $\frac{C}{2}+1$.
\end{proof}

Similar to the definition in $H$ we define a simple vertex of $G$ to be a vertex in the cycle $K$ such that this vertex does not lie in any small cycle of $G$. We now prove that every vertex of degree 2 corresponding to a common card of $G$ and $H$ is simple.

\begin{Lemma}\label{again simple vertices}
    Every vertex of degree 2 in $G$ corresponding to a common card is simple.
\end{Lemma}
\begin{proof}
    Suppose $v$ is a vertex of degree 2 in $G$ corresponding to a common card, but $v$ is not simple. Then $v$ lies in a small cycle $C$ (and a large cycle). Since $v$ has degree 2, both neighbors of $v$ also lie in $C$ and $K$.  Let $v_1$ and $v_2$ be the vertices in $K \cap C$ that have the largest distance between them in $K$.\\
    Because $C \leq \frac{n}{3}-4$ and every distance in $K$ is smaller than the distance in $C$, we obtain $d_{K}(v_1,v_2) \leq \frac{n}{6}-2$. Moreover, there are no other vertices on the long arc $v_1v_2$ of $K$ that also lie in $C$ as those vertices will have a larger distance to either $v_1$ or $v_2$ than $d_{K}(v_1,v_2)$. So $v$ lies on the short arc of $K$. If we now replace the short arc $v_1v_2$ in $K$ with the arc $v_1v_2$ of $C$ without $v$, we get a cycle that does not contain $v$. This new arc of $C$ is larger than the short arc in $K$ implying that this cycle is larger than $K$ and thus not small. Since $H$ does not contain any large even cycles and $G-v$ contains a large even cycle, $G-v$ is not a card of $H$.
\end{proof}

We will now combine all the information from the previous lemmas to derive a contradiction.

\thmbipartite*
   \begin{proof}  
We use proof by contradiction. Suppose there exists a bipartite graph $G$ and a non-bipartite graph $H$ that have a subdeck $D$ of at least $\frac{5n}{6}+2$ common cards.
   
 By Lemmas \ref{lotsdeg2} and \ref{again simple vertices} at least $\frac{2n}{3}+4$ common cards in $D$ correspond to a  simple vertices in $H$ and $G$ of degree 2. Hence by the proof of Lemma \ref{same number of edges}, we can determine the number of edges in $G$ and $H$ based on $D$ and thus we can determine the degrees of the missing vertices 
 for all cards in $D$.

We will now prove that $K-1$ and $L-1$ are both equal to the smallest longest path of all common cards corresponding to a vertex of degree 2. We will prove this statement for $K$ as the proof for $L$ will be the same.

We start by observing the following. The neighbors of a simple vertex $v$ must have distance $K-1$ in $G-v$ as otherwise, $v$ would lie in a smaller cycle than $K$. Hence, the diameter on all cards corresponding to a vertex of degree 2 is at least $K-1$.

We will use a similar coloring as in Figure \ref{Kleuring 1}. This means that for every branch of size $k$ outside our fixed cycle $K$ connected to a vertex $v$ in $K$, we color all vertices on $K$ with distance at most $k$ to $v$ red. Since the vertices of a path of length $m$ outside the cycle are only connected to vertices that lie in an interval of size $m+1$ on the cycle by Lemma \ref{interval}, we color at most $3k+2 \leq 5k$ vertices red for every branch of size $k$. 

This indicates that we color at most $5(n-K)$ vertices on $K$ red. Hence, there are at least \[K-5(n-K)= 6K-5n=6\left(K-\frac{5n}{6}-2\right)+12\] white vertices on the cycle. Since at most $K-\frac{5}{6}n-2$ vertices on the cycle of $G$ do not correspond to a common card, there exists a white vertex $v$ corresponding to a card in $D$. By our coloring, we know that this vertex has degree 2 and is therefore simple. 

 Next, we show that for every white vertex $v$ the diameter of $G-v$ is exactly $K-1$. Suppose not. Let $x$ and $y$ be two vertices in $G-v$. Let $v_x$ and $v_y$ be the vertices closest to $x$ and $y$ such that $v_x$ and $v_y$ lie on the cycle $L$ and such that $v_x$ lies closer to $v_1$ and $v_y$ to $v_2$. Then as $v$ is white, we see $d(x,v_x) < d(v,v_x)$ and $d(y,v_y)<d(v,v_y)$ implying that \[d(x,y) \leq d(x,v_x)+d(v_x,v_y)+d(v_y,y) \leq d(v_1,v_x)+d(v_x,v_y)+d(v_y,v_2) =  d(v_1,v_2)=K-1,\]
where the second to last equality follows from the fact that the unique path between $v_1$ and $v_2$ is the path corresponding to the large arc of the cycle.

Hence, $K-1$ is equal to the minimal diameter of (common)cards $G-v$, where $v$ has degree 2. Similarly, we can prove that $L-1$ is equal to the minimal diameter of (common)cards $H-w$, where $w$ has degree 2.
For a common card $G-v=H-w$, the vertex $v$ has degree 2 if and only if $w$ has degree 2. So in both cases, we take the minimum over the same set of cards. Thus, $K-1=L-1$ implying $K=L$, a contradiction.
\end{proof}

So we have proved an upper bound of $\frac{5n}{6}+2$ for determining whether a graph on $n$ vertices is bipartite. Now, we consider possible lower bounds. Notice that if we take an element of Family 1 with $n \equiv 1 \mod 4$ vertices, then $G$ contains an odd cycle of length $\frac{n+1}{2}$, while $F$ does not contain such a cycle. This means that we found an infinite family of pairs of bipartite graphs $F$ and non-bipartite graphs $G$ that have $\lfloor \frac{n}{2}\rfloor+1$ common cards. 

Moreover, for a member of Family 3 on $n \equiv 1 \mod 4$ vertices, the graph $G$ has an odd cycle of length $\frac{n+1}{2}$, while $F$ does not contain any cycles. This gives our second infinite family of pairs of bipartite graphs $F$ and non-bipartite graphs $G$ that have $\lfloor \frac{n}{2}\rfloor+1$ common cards.

Figure \ref{fig:bipartiet}contains another infinite family of pairs of graphs $G$ and $H$, where $G$ is bipartite and $H$ is not and such that $G$ and $H$ have $\frac{n+1}{2}$ common cards. What is special about this example is the fact that $G$ and $H$ both contain cycles and that both are connected. We can extend this example to a whole family by adding extra diamonds to $G$ and $H$. Since every diamond increases the large cycle of $H$ by $3$, $H$ is not bipartite if it contains an even number of diamonds. However, $G$ is always bipartite so this gives indeed an infinite family of pairs.

\begin{figure}[ht]
    \centering
    \includegraphics[width = 15 cm]{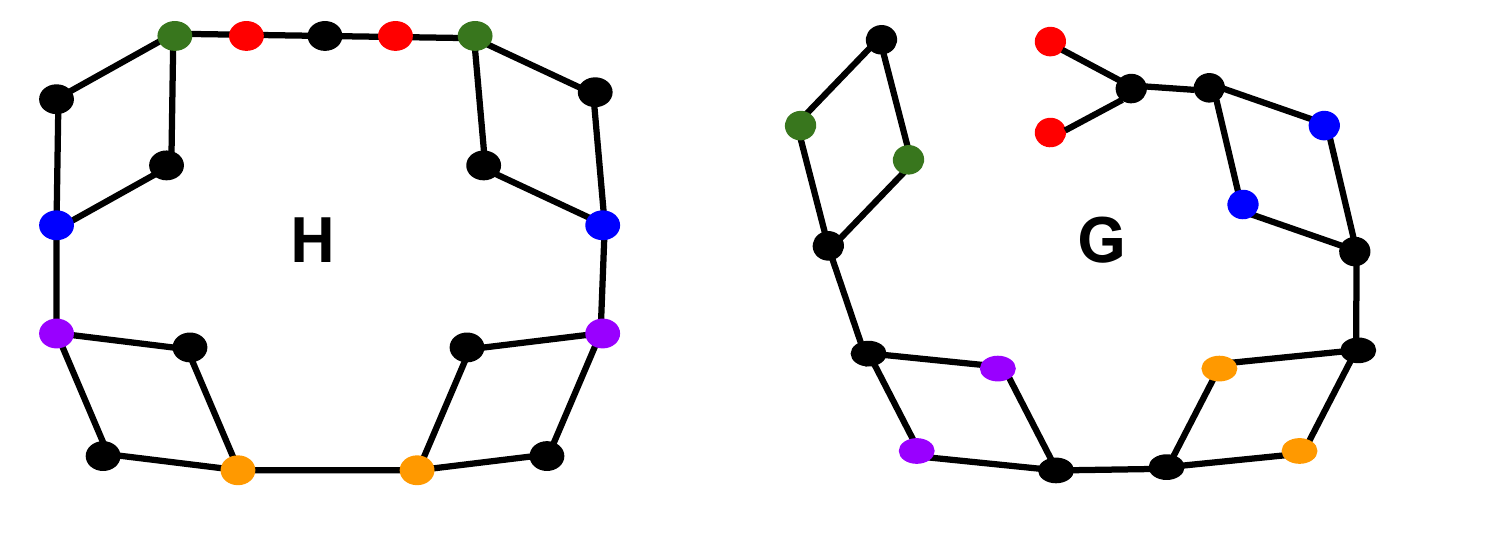}
    \caption{Another example of a bipartite graph $G$ and an non-bipartite graph $H$ that have $\frac{n+1}{2}$ common cards. For every two vertices $v \in G$ and $w \in H$ with the same color (not black), we have $G-v \cong H-w$.}\label{fig:bipartiet}
\end{figure}

\section{Discussion}

In this article, we have developed some machinery to determine whether a graph has a `big' hidden cycle if we have $n-\epsilon n$ arbitrary cards of a graph $G$ on $n$ vertices. The main idea was to mark vertices that lie close to big branches or other cycles as bad (here always colored red). This means that for good vertices $v$, most of the vertices that lay in the `big hidden' cycle are lying in the longest path. This also implies that on every card it is always apparent what most of the branches/small cycles around the big cycle in $G-v$ are.

However, in order to use this information, the subdeck must contain at least one common card corresponds to a good vertex. This requirement led to different bounds for different properties. We were for example able to show that we can determine whether a connected graph on $n$ vertices has a cycle based on any $\frac{n}{2}-\epsilon \sqrt{n}$ cards, whereas we needed at least $\frac{2n}{3}+2$ cards to determine the length of a smallest cycle. Moreover, if we wanted to determine whether a graph had an odd cycle we needed at least $\frac{5n}{6}+2$ cards, as we had to ensure that there was a big gap between `large' and `small cycles' as otherwise our coloring would not work.

We believe that these colorings can also be used to determine other properties that involve the (non-)existence of a few cycles based on any $n-\epsilon n$ cards. The main idea is to choose enough cards such that every small cycle is at least once visible.  Hence, we only need to focus on the question of whether there is a big hidden cycle. For example, we believe it is possible to determine whether a graph has at most one cycle or more than one cycle based on any $n-\epsilon n$ cards. However, proving this requires taking into account the possibility that there are two small cycles of the same length. 

Although we have found upper bounds for girth and bipartiteness of the form of $n-\epsilon n$, we suspect that the true upper bounds are far lower. We believe that the lower bounds of $\lfloor \frac{n}{2}\rfloor+2$ are in fact the real bounds. Therefore, we propose the following conjectures.

\begin{Con}
    For $n$ large enough, the number of common cards between a graph $G$ and $H$, both on $n$ vertices, with different girth is at most $\lfloor \frac{n}{2}\rfloor+1$.
\end{Con}

\begin{Con}
    For $n$ large enough, the number of common cards between a bipartite graph $G$ and a non-bipartite graph $H$ that both have $n$ vertices is at most $\lfloor \frac{n}{2}\rfloor+1$.
\end{Con}

Moreover, it would be very interesting to find all graphs (for $n$ large enough) for which these bounds are tight. For example, Bowler, Brown, Fenner, and Myrvold \cite{bowler2011recognizing} characterized all pairs of graphs $(G, H)$ where $G$ is connected and $H$ is disconnected that have $\lfloor \frac{n}{2} \rfloor+1$ common cards. In \cite{Masterscriptie} there is a proof that, for $n$ large enough, all pairs of graphs $(G, H)$ where $G$ is a forest and $H$ is not a forest that have exactly $\lfloor \frac{n}{2}\rfloor+1$ common cards, are exactly theF ones that fall in one of the three the families of Observation \ref{three families}. Therefore, we pose the following question:

\begin{Q}
    What are all infinite families of pairs of graphs $G$ and $H$ both on $n$ vertices, such that $G$ is bipartite and $H$ not, that have (at least) $\lfloor \frac{n}{2}\rfloor+1$ common cards?
\end{Q}

\begin{acknowledgement} \normalfont
The article is based on results from my master's thesis \cite{Masterscriptie} under the supervision of Carla Groenland. I would like to thank her for our useful discussions and her mentorship. Moreover, her valuable suggestions heavily improved the outline and structure of this article.
\end{acknowledgement}

\begin{open} For the purpose of open access, a CC BY public copyright license is applied to any Author Accepted Manuscript (AAM) arising from this submission.
\end{open}

\bibliographystyle{plain}
\bibliography{references}

\appendix
\section{Deferred proofs}\label{appendix: forest}
This section of the appendix contains some deferred results and proofs that were earlier.

\subsection{Components of cards of connected graphs}

\begin{Lemma}\label{component H in connected} Let $H,G$ be connected graphs such that $v(H) < \frac{v(G)}{2}$. Then at most $\left\lfloor\frac{v(G)}{v(H)+1} \right\rfloor$ cards of $G$ contain a connected component isomorphic $H$.
\end{Lemma}
\begin{proof} If there is no card of $G$ that contains a connected component isomorphic to $H$ we are immediately done. So let $v$ be a vertex such that $H$ is a connected component of $G-v$. We consider the set $V(H_{G})$ of vertices in $G$ that lie in this component.  As $G$ is connected, there is a vertex $w_1 \in V(H_G)$ and $x \in G-v-V(H_G)$ such the edges $(w_1,v)$ and $(v,x)$ exist.

Given $w_i \in V(H_G)$, we will show that the graph $G-w_i$ does not contain a connected component $H$. The graph $G\slash V(H_G)$ is connected as $G$ is connected and $H$ is a connected component in $G-v$. This means that the component in $G-w_i$ containing $v$ has at least $v(G)-v(H_G) > v(H_G)$ vertices and is thus not isomorphic to $H$. All other  connected components are a subset of $V(H_G)-w_i$ which has strictly fewer vertices than $H$. Thus these components are also not isomorphic to $H$. This means that for all vertices $w_i \in V(H_G)$, the graph $G-w_i$ does not contain a connected component isomorphic to $H$.

Let $u \neq v$ and $G-u$ be another card with a connected component $H'$ which is isomorphic to $H$. By the previous paragraph, we see that $u \neq w_i$. We will show that $V(H) \cap V(H'_G) = \emptyset$. As $H$ is a connected component in $G-v$ and $u \notin V(H_G)$, we find that in $G-u$ the vertex $v$ is in the same connected component as $H$. This means that all vertices of $H$ lie in a connected component with more vertices than $H$, so none of the vertices in $V(H_G)$ are also in $V(H'_G)$.

This means that for every vertex $v$ such that $G-v$ contains a connected component of size $k$, there are uniquely defined $v(H)$ vertices for which $G-w$ does  not contain such a component. If $|C|$ is the number of cards that have $H$ as a connected component, then we see that $|C|\cdot (v(H)+1) \leq v(G)$, which concludes the proof.
\end{proof}

\begin{Remark}
The bounds of Lemma \ref{component H in connected} are sharp. Draw a cycle $C$ with $v(G)- k\left\lfloor\frac{v(G)}{k+1}\right\rfloor$ vertices and add then a copy of H to $\left\lfloor\frac{v(G)}{k+1}\right\rfloor$ vertices in this cycle $C$. Then for each of those $\left\lfloor\frac{v(G)}{k+1}\right\rfloor$ vertices, we have that $G-v$ contains a connected component isomorphic to $H$.
\end{Remark}

In the proof of Lemma \ref{component H in connected}, we only considered the number of vertices of the components of $G-v$ and never whether a connected component of size $k$ was isomorphic to $H$. This  means that we have proved the following even stronger statement.
\begin{Cor}\label{concom}
Let $G$ be a connected graph and let $1 \leq k < \frac{v(G)}{2}$. Then there are at most $\left \lfloor \frac{v(G)}{k+1}\right\rfloor$ vertices $v \in G$ such that $G-v$ contains a connected component of size exactly $k$.
\end{Cor}

\subsection{Upper bound common cards between forests and non-forests}
\forestvscycle*

\begin{proof}
   We will use a proof by contradiction. Suppose that a forest $T$ and a graph $G$ with a cycle,  both on $n$ vertices, have at least $\frac{2n}{3}+1$ common cards. Let $L$ be the length of a cycle of $G$, then $L \geq \frac{2n}{3}+1$ as the cycle is not  visible on any of the common cards. Now, Lemma \ref{max een cycle} tells us that $G$ is unicyclic, so let $L$ be the length of the unique cycle. Then as $G$ is unicyclic, $G$ has at least $L-(n-L) = 2L-n= 2(L-\frac{2n}{3}-1)+\frac{n}{3}+2$ vertices of degree $2$ on its cycle. As at most $L - (\frac{2n}{3}+1)$ vertices on the cycle of $G$ do not correspond to a common card, at least $\frac{n}{3}+2$ common cards do correspond to a vertex of degree 2.   

   Let $K_G$ be the number of connected components in $G$ and $K_T$ the number of connected components of $T$. Then $e(G)= n + 1 - K_G$ and $e(T) = n -K_T$.
   For every common card corresponding to a vertex of degree 2 in the cycle of $G$, the number of connected components is equal to $K_G$.

   Therefore, $K_T \leq K_G+1$. If $K_T=K_G+1$ then $T$ must have at least $\frac{2n}{3}+1$ isolated vertices, as deleting them is the only way to decrease the number of connected components. This indicates that $K_G \geq \frac{2n}{3}+1$. However, as the cycle of $G$ already has $\frac{2n}{3}+1$ vertices, $G$ also has a connected component with at least $\frac{2n}{3}+1$ vertices, which leads to a contradiction. Therefore, $K_T \leq K_G$ and $e(G) > e(T)$. 

   Therefore, for every common card $T-v=G-w$ with $\deg(w)=2$, the fact $e(G)>e(T)$ implies $\deg(v) \leq 1$. This indicates that $T$ has at least $\frac{n}{3}+2$ vertices of degree at most $1$. Then $T-v=G-w$ has at least $\frac{n}{3}+1$ such vertices. As $w$ has degree 2, this indicates that $G$ has at least  $\frac{n}{3}-1$ vertices of degree at most 1 as this number only decreases if a neighbor of $w$ had degree 1 or less. But as $G$ already has $\frac{2n}{3}+1$ vertices on the cycle, which all have degree at least 2, all bounds must be tight. 
   
   In particular, this means that $L=\frac{2n}{3}+1$, and thus every vertex on the cycle of $G$ corresponds to a common card. Moreover, this also tells us that every vertex of degree $2$ in the cycle has two neighbors of degree at most 2. Hence, all vertices in the cycle have we degree 2.  Therefore, $T$ has at least $\frac{2n}{3}+1$ vertices of degree at most 1 and $G$ has $\frac{2n}{3}-2$ vertices of degree at most 1. As $G$ has n vertices, this gives us that $n \geq \frac{2n}{3}-2+\frac{2n}{3}+1$, implying $3 \geq n$.

   On the other hand, we also know that $\frac{2n}{3}+1 \leq n$, so $n=3$. In this case, $G$ is a triangle and all 3 cards of $G$ are connected. However, as all vertices of $T$ have degree at most 1, we see that $T$ is disconnected. So $T$ does not have 3 connected cards, a contradiction.
\end{proof}
\end{document}